\documentclass[reqno, 11pt,a4paper]{amsart}

\usepackage{fullpage}  
\usepackage{calc,graphicx,amsfonts,amsthm,amscd,epsfig,psfrag,amsmath,amssymb,enumerate,dsfont}
\usepackage{subfig} 
\usepackage{pdfsync}
\usepackage{stmaryrd} 
\usepackage[initials]{amsrefs}

\usepackage[colorlinks=true, pdfstartview=FitV, linkcolor=blue, citecolor=blue, urlcolor=blue,pagebackref=false]{hyperref}

\usepackage[showlabels,sections,floats,textmath,displaymath]{}
\reversemarginpar
\newlength\fullwidth
\numberwithin{equation}{section}

\DeclareMathSymbol{\leqslant}{\mathalpha}{AMSa}{"36} 
\DeclareMathSymbol{\geqslant}{\mathalpha}{AMSa}{"3E} 
\DeclareMathSymbol{\eset}{\mathalpha}{AMSb}{"3F}     
\renewcommand{\leq}{\;\leqslant\;}                   
\renewcommand{\geq}{\;\geqslant\;}                   

\newcommand{\eps}{\epsilon}

\newcommand{\1}{\mathds{1}}

\renewcommand{\l}{\lambda}
\renewcommand{\L}{\Lambda}

\renewcommand{\l}{\lambda}

\renewcommand{\O}{\Omega}

\newtheorem{theorem}{Theorem}[section]

\newtheorem{lemma}[theorem]{Lemma}
\newtheorem{proposition}[theorem]{Proposition}

\newtheorem{remark}[theorem]{Remark}

\newtheorem{definition}[theorem]{Definition}

\newcommand{\cD}{\ensuremath{\mathcal D}}

\newcommand{\cL}{\ensuremath{\mathcal L}}

\newcommand{\cN}{\ensuremath{\mathcal N}}

\newcommand{\cP}{\ensuremath{\mathcal P}}

\newcommand{\cU}{\ensuremath{\mathcal U}}

\newcommand{\bbE}{{\ensuremath{\mathbb E}} }

\newcommand{\bbH}{{\ensuremath{\mathbb H}} }

\newcommand{\bbP}{{\ensuremath{\mathbb P}} }

\newcommand{\bbR}{{\ensuremath{\mathbb R}} }

\newcommand{\bbZ}{{\ensuremath{\mathbb Z}} }

\newcommand{\E}{\mathbb{E}}

  \let\h=\eta      \let\l=\lambda

     \let\L=\Lambda 
\let\O=\Omega      

\renewcommand{\le}{\leq}

\newcommand{\Var}{{\rm Var}}

\newcommand{\ip}[2]{\langle #1,#2\rangle}

\title{Linear time mixing and pre-cutoff for oriented kinetically constrained models}
\author[P. Chleboun]{P. Chleboun$^1$}
 \address{$^1$Department of Statistics, University of Warwick,  Coventry,  CV4 7AL,  United Kingdom.}
\email{paul.i.chleboun@warwick.ac.uk}

\date{}

\begin{document}

\begin{abstract}
  We establish linear time precutoff for oriented kinetically constrained models (KCMs) on a $d$-dimensional box of side length $n$, whenever the finite-volume spectral gap is bounded away from zero uniformly in $n$. 
  A typical example is the North-East model, a $0$-$1$ spin system on the two-dimensional integer lattice that evolves according to the following rule: whenever a site's southerly and westerly nearest neighbours have spin $0$, with rate one it resets its own spin by tossing a $(1-q)$-coin; at all other times its spin remains frozen. 
  This settles, in the \emph{oriented case}, a conjecture which states that there is linear time precutoff for all KCMs in the (open) ergodic parameter regime ($q>q_c$).
  The result has been shown recently, for \emph{general} update families, in a perturbative regime ($q$ close to $1$). The previous general upper bound on the mixing time for oriented update families, in the full ergodic regime, was $O(n \log n)$.
  The proof combines a spectral estimate for a certain killed process with a graphical coupling to a stationary process and a backward path argument, exploiting the orientation of the constraints, to control the probability of failing to couple.\\


\end{abstract}

\maketitle

\section{Introduction}

Kinetically constrained models (KCMs) are interacting $0$-$1$ particle (spin) systems which were introduced in the physics literature to model the dynamics of glassy systems \cites{FredricksonAndersen1984,Ritort}. See also \cite{HTbook}*{Chapter 1} and references therein for a much more detailed discussion of the history and substantial literature.

The non-conservative KCMs considered here are Markov processes which evolve according to a Glauber-type dynamics with a dynamic constraint.
Independently at rate one each site attempts to update its spin by tossing a $(1-q)$-coin, the update is allowed if and only if the configuration of neighbouring sites satisfies a certain local constraint, otherwise the update is rejected.
The constraint requires that all sites in at least one of the sets of neighbours specified by a fixed update family, are in the $0$ state (facilitating). 
This constraint mimics the ``caging'' effect observed in glassy and amorphous systems.
These models capture many characteristic features of glassy and amorphous materials, such as dynamic facilitation, dynamic heterogeneity and extremely slow relaxation to equilibrium for certain parameter values, while remaining reversible with respect to the product Bernoulli measure.
Stochastic dynamics with a facilitation mechanism also appear in contexts unrelated to glassy systems, such as random walks on upper triangular matrices \cite{Ganguly} and Glauber dynamics of solid-on-solid interfaces and monotone surfaces \cite{Caputo12}.

Despite the simplicity of their definition, and reversibility, the degenerate rates of KCM make them challenging to analyse mathematically, largely due to the lack of monotonicity, complex cooperative effects in the dynamics and sometimes exhibiting ergodicity-breaking phase transitions.
In recent years a detailed understanding of KCMs at equilibrium, in particular in two dimensions, has been achieved \cite{HTbook}*{Chapters 1--6}.
However, because of the mathematical challenges, the understanding of the non-equilibrium dynamics is much more limited (see \cite{HTbook}*{Chapter 7} and references therein).

Results on the non-equilibrium dynamics are mostly restricted to the East model, where the constraint requires a single nearest neighbour in a fixed direction to be in state $0$. 
For this process several model specific tools greatly simplify the analysis.
Linear time mixing in arbitrary dimensions has been establish for the East model across the parameter range $p\in(0,1)$, with cutoff established in one-dimension \cites{CFM15,GLM15}. 
Cutoff has also been established for the East model in higher dimensions with certain boundary conditions in the small-$q$ \cite{CouzinieMartinelli2022} and large-$q$ regimes \cite{CampaillaMartinelli2025}.
Results are also available for the FA-1f model, where the constraint requires at least one nearest neighbour to be in state $0$ (without a direction), for certain values of the parameter $q$, see \cites{BCMRT13,BDT19,Ertul22} and references therein.

For more general update families, results are largely restricted to the perturbative regime $q$ close to $1$, for which there has recently been significant progress, see \cite{HflT} and references therein.
There, local exponential convergence to equilibrium in infinite volume, started from a suitable product measure, and linear-time pre-cutoff in finite boxes with facilitating boundary conditions, have been established for arbitrary update families in the perturbative regime $q$.

In this paper we consider oriented update families, for which there exists a direction such that all sites in the constraint set are on one side of a hyperplane through the site, orthogonal to a fixed direction.
We examine the total-variation mixing time in a box $\L_n=\{1,\dots,n\}^d$ of side length $n$ with facilitating boundary conditions, with the update rule and $q \in (0,1)$ fixed as $n$ diverges. 
Such directed models are one of the few cases where non-perturbative non-equilibrium results have previously been established. 
For this class, an $O(n \log n)$ upper bound strictly within the ergodic regime ($q>q_c$) was established in \cite{CM13}.
Starting from a configuration of all ones, finite speed of propagation gives a lower bound of order $n$ (see \cite{HTbook}*{Proposition 3.12}).
In Theorem \ref{th:main} we establish the matching upper bound of order $n$ for the mixing time, and hence linear-time pre-cutoff, throughout the entire ergodic regime $q>q_c$.
In particular this settles \cite{HTbook}*{Conjecture 7.30} for all oriented update families.

For the North-East model, it has been conjectured \cite{Kordzakhia:2006} that for $q>q_c$, the set of sites which have been updated at least once by time $t$ grows linearly with $t$ and when correctly rescaled converges to a deterministic asymptotic limit shape.
Also, the process as seen by the front is expected to be ergodic, and hence proving such a limit shape and associated CLT may provide a path to cutoff (see \cite{HflT}*{Section 4} where this path is explained in more detail).
The pre-cutoff result here can be viewed as a step towards these much more ambitious goals.

The proof strategy is to consider a version of the KCM dynamics which is killed at the first legal ring at a fixed site.
This killed process is shown to be contracting in $L^2$ with respect to the stationary measure.
This allows us to control the probability of an arbitrary chronological sequence of failed legal-ring intervals under the stationary process, across sites in the box, without having to deal directly with the dependency structure.
Using the standard graphical construction coupling, we consider a monotone increasing sequence of random sets describing the sites that we know must have permanently coupled with the stationary process uniformly over initial configurations. 
Consider a site which has not coupled after many fixed $O(1)$ time intervals.
During each time interval, the failure of a site to couple must be due to either a failure to have previously coupled at least one site in the constraint set, or a failure to have a legal ring \emph{in the stationary process}.
By considering a suitable witness path which moves towards the boundary if there was a site that previously had not coupled, then 
a failure to couple after a large multiple of $n$ time intervals implies a failure of order $n$ legal rings in the stationary process.
This is exponentially unlikely by the contraction of the killed process.
A union bound over all possible witness paths oriented towards the boundary gives the desired upper bound.

\section{Models and main results}

\subsection{Notation and preliminaries}

Let $\Lambda_n=\{1,\dots,n\}^d$ be the $d$-dimensional cube of side $n$ in $\bbZ^d$.
The standard basis vectors of $\bbZ^d$ are denoted by $e_1=(1,\ldots,0),\dots,e_d=(0,\ldots,1)$. For $x\in \bbZ^d$ we write $x=(x_1,\dots,x_d)$.

The set of probability measures on a measurable space $\Omega_n = \{0,1\}^{\Lambda_n}$ is denoted by $\cP(\Omega_n)$. 
Elements of $\Omega_n$ are denoted by lower case Greek letters $\sigma,\eta,\ldots$ and are called configurations. For a configuration $\sigma\in \Omega_n$ and $x\in \Lambda_n$, we write $\sigma_x$ for the value of the configuration at site $x$.

We refer to unit vectors 
\[
u \in S^{d-1} = \left\{ v \in \bbR^d\,:\, \|v\| = 1 \right\}
\]
as directions. For a direction $u\in S^{d-1}$, we define the integer half-space $\bbH_u = \{ x \in \bbZ^d : \langle x,u \rangle< 0 \}$.
The KCM dynamics are defined with respect to an \emph{update family}.
An update family $\cU$ is a finite non-empty collection of finite non-empty subsets of $\bbZ^d\setminus\{0\}$, called update rules. This update family is called \emph{oriented} if there exists a direction $u\in S^{d-1}$ such that for all $U\in \cU$, $U\subset \bbH_u$ (such a direction is typically called \emph{unstable}).

\subsection{The KCM dynamics}

We define the $\cU$-KCM dynamics on $\Omega_n$  in terms of the following graphical construction, which simultaneously couples the dynamics for all initial configurations. 
Fix $p\in(0,1)$ and let $q=1-p$.
For each $x\in \Lambda_n$, let $T_x$ be an independent rate one Poisson process on $[0,\infty)$.
We mark each Poisson process with a collection, $\{S_{x}(t)\}_{t\in T_x}$, of i.i.d. Bernoulli($p$) random variables (coin tosses), independent of the Poisson processes (and each other). 
The probability measure and expectation associated with this graphical construction are denoted by $\bbP$ and $\bbE$, respectively.

Since the constraints may depend on sites outside of $\Lambda_n$, we define the dynamics on the extended configuration space $\Omega = \{0,1\}^{\bbZ^d}$, with the convention that sites outside of $\Lambda_n$ are fixed equal to $0$ at all times (i.e., they are always facilitating).
To reduce notation we do not explicitly distinguish between configurations in $\Omega_n$ and their extensions to $\Omega$.

For an initial configuration $\eta \in \Omega_n$, the $\cU$-KCM dynamics $\{\eta(t)\}_{t\ge 0}$ is defined as follows. At each time $t\in T_x$, if the configuration $\eta(t^-)$ satisfies the constraint at $x$ (i.e., there exists $U\in \cU$ such that $\eta(t^-)_y=0$ for all $y\in x+U$), then the spin at $x$ is updated to $S_{x}(t)$; otherwise, it remains unchanged.
Such update times are called  \emph{legal rings} at $x$.

Define the constraint function at site $x$ by 
\begin{align}
    \label{eq:constraint}
    c_x(\eta) = \1_{\{\exists U\in \cU : \eta_y=0 \text{ for all } y\in x+U\}}\,,
\end{align}
then, we have
\begin{align}
    \label{eq:dynamics}
    \eta_x(t) = \begin{cases}
        S_{x}(t) & \text{if } t\in T_x \text{ and } c_x(\eta(t^-))=1\,,\\
        \eta_x(t^-) & \text{otherwise.}
    \end{cases}
\end{align}
Denote the law of the process started from $\eta$ by $\bbP_\eta$ and the corresponding expectation by $\bbE_\eta$. Hence, $\bbP_\eta(\eta(t)\in \cdot) = \bbP(\eta(t)\in \cdot\mid \eta(0)=\eta)$. Furthermore, for $\mu \in \cP(\Omega_n)$, denote the law of the process started from $\mu$ by $\bbP_\mu$ and the corresponding expectation by $\bbE_\mu$.

Note that, since $c_x(\eta)$ does not depend on $\eta_x$, it is straightforward to see that the dynamics are reversible with respect to the product Bernoulli($p$) measure $\pi_n= \textrm{Ber}(p)^{\otimes \L_n}$ on $\Omega_n$.

For oriented update rules, the above construction defines an irreducible continuous-time Markov chain on $\Omega_n$ with unique stationary distribution $\pi_n$ and generator $\cL_n$ given by
\begin{align}
    \label{eq:generator}
    \cL_n f(\eta) = \sum_{x\in \Lambda_n} c_x(\eta) \left[\pi_x( f)(\eta) - f(\eta) \right]\, \quad \text{for all } f:\Omega_n\to \bbR\,,
\end{align}
where $\pi_x(f)(\eta)$ denotes the conditional mean $\pi_n(f\mid  \{\eta_y\}_{y\neq x}) = p f(\eta^{x\leftarrow 1}) + q f(\eta^{x\leftarrow 0})$ with $\eta^{x\leftarrow i}_x = i$ and $\eta^{x\leftarrow i}_y = \eta_y$ for $y \neq x$, and $q = 1-p$.

We finish this section with definitions of the spectral gap and
mixing time of the process.
\begin{definition}[spectral gap]
\label{gap}
The \emph{spectral gap}, $\l^{(n)}$, of the infinitesimal generator \eqref{eq:generator} is the smallest positive eigenvalue of $-\cL_n$ , and is given by the variational
principle
\begin{align}
  \label{eq:5}
  \l^{(n)} : = \inf_{\substack{f:\O_n \to \bbR \\ f \neq \textrm{const}}}\frac{\cD_n(f)}{\Var_{\pi_n}(f)},
\end{align}
where $\cD_n(f)= -\pi_n\left(f\cL_n f\right)$ is the Dirichlet form of
the process.
\end{definition}

\begin{remark}
    The process is well-defined on the infinite lattice, and we denote the associated spectral gap by $\l$. 
    It turns out that $\l = \inf_n \l^{(n)} = \lim_{n \to \infty} \l^{(n)}$ \cite{HTbook}*{Proposition 3.11},\cite{CMRT}*{Lemma 2.11, Proposition 2.13}.
    The exponential decay critical parameter is given by $\tilde{q}_c(\cU) = \inf\{q>0\,:\, \l > 0\}$.
    If we define $\tau_0$ to be the first time the origin is updated in the infinite volume process, then the ergodicity critical parameter is given by $q_c(\cU) = \inf\{q>0\,:\, \bbP_\pi(\tau_0 <\infty) = 1\}$, where $\pi$ is the infinite volume product Bernoulli measure with parameter $1-q$. 
    For \emph{oriented} KCM $\tilde{q}_c(\cU) = q_c(\cU)$, see for example \cite{Hartarskybootstrap}*{Corollary 1.8} and \cite{HTbook}*{Section 7.1}. 
\end{remark}
\begin{definition}[Total variation mixing time]
For $t \geq 0$, let
\begin{align}
  \label{eq:dn}
 d_n(t):=\sup_{\eta\in\Omega_n}\|\bbP_\eta(\eta(t) \in \cdot)-\pi_n\|_{\rm{TV}},
\end{align}
where 
\[
\|\mu-\nu\|_{\rm{TV}} = \sup_A|\mu(A)-\nu(A)| = \frac{1}{2}\sum_{\eta\in\Omega_n}|\mu(\eta)-\nu(\eta)|\,.
\] 
Then, the \emph{total variation mixing time} of the $\cU$-KCM on $\Lambda_n$ is defined by
\begin{align}
  \label{eq:tmix}
 t_{\mathrm{mix}}^{(n)}(\varepsilon):=\inf\{t \geq 0\,:\,d_n(t)\leq \varepsilon\}.
\end{align}
\end{definition}

\subsection{Main results}
\begin{theorem}
\label{th:main}
Let $\cU$ be an oriented update family and assume $\inf_n \l^{(n)} > 0$, which holds for  $q\in(q_c(\cU),1)$. Then the mixing time of the $\cU$-KCM on $\Lambda_n$, with facilitating boundary conditions, exhibits pre-cutoff in linear time. More precisely, there exists a constant $C=C(q,d,\cU)\geq 1$ such that, for all $\varepsilon \in (0,1)$ and all $n$ sufficiently large (depending on $\varepsilon$)
\begin{align}
    \label{eq:main}
    C^{-1} n \leq t_{\mathrm{mix}}^{(n)}(\varepsilon) \leq C n\,.
\end{align}
\end{theorem}

\begin{remark}
   Recall, a sequence of Markov chains is said to exhibit pre-cutoff if
   \[
   \sup_{\eps \in (0,1/2)}\limsup_{n \to \infty}\frac{t_{\mathrm{mix}}^{(n)}(\eps)}{t_{\mathrm{mix}}^{(n)}(1-\eps)} < \infty\,,
   \]
   see for example \cite{Levin2008}. Theorem \ref{th:main} shows that the ratio is bounded above by a constant $C^2$ independent of $\eps$.
\end{remark}

\begin{remark}
   The restriction to $q \geq q_c(\cU)$ is necessary for the result to hold in the North-East model. 
   Indeed, for $q<q_c(\cU)$ the North-East process has a relaxation time, and hence mixing time, which is exponentially large in $n$ \cite{CMRT}*{Theorem 6.16}.
   The same is expected to hold for general nontrivial oriented models.
   A lower bound on the relaxation time should follow by adapting the arguments in \cite{CMRT}, since there exist macroscopic stable regions under the stationary distribution by comparison with the associated bootstrap percolation process \cite{HTbook}*{Chapter 7}.
   The critical case, $q=q_c(\cU)$, remains an open problem.
\end{remark}

\section{Proof of Theorem \ref{th:main}}

\subsection{Oriented KCM}
Recall the update family $\cU$ is a finite non-empty collection of finite non-empty subsets of $\bbZ^d\setminus\{0\}$, i.e.
\[
   \cU=\{U_1,\ldots,U_m\},
   \qquad \varnothing\ne U_i\subset \bbZ^d\setminus\{0\}\,.
\]
Write
\[
   \cN:=\bigcup_{U\in\cU}U,
   \qquad N:=|\cN|.
\]
The set $\cN$ is the full dependency neighbourhood. 

\begin{definition}[Oriented KCM]
\label{def:oriented}
The update family $\cU$ is \emph{oriented}, and the corresponding KCM is \emph{oriented}, if there exists a direction $v\in S^{d-1}$ such that $\cup_{U\in\cU}U \subset \bbH_v$. 
Equivalently, there is a $v\in S^{d-1}$ such that
\[
       \ip{u}{v}<0 \qquad\text{for every }u\in\cN.
\]
For one such $v$, set
\[
   \delta_v:=\min_{u\in\cN}\bigl(-\ip{u}{v}\bigr)>0.
\]
\end{definition}
Define the \emph{height} of a site $x\in\bbZ^d$ in direction $v$ by $h(x):=\ip{x}{v}$. 
The direction $v$ will be fixed throughout given the update family $\cU$, and hence we will not include it in the notation for the height.

\subsection{The killed process}

Fix $x \in \L_n$ and an initial configuration $\eta \in \O_n$, let $(\eta(t))_{t\ge 0}$ be the $\cU$-KCM dynamics on $\O_n$ started from $\eta$.
Define the \emph{first legal ring} at $x$ by
\[
   \tau_x:=\inf\{t\ge 0\,:\, t \in T_x \text{ and } c_x(\eta(t^-))=1\}.
\]
For $t\ge 0$, define the process killed at the first legal ring at $x$ by the substochastic semigroup
\begin{equation}\label{eq:Qdef}
     (Q_t^x f)(\eta):=\E_\eta\bigl[f(\eta(t))\1_{\{\tau_x\ge t\}}\bigr].
\end{equation}
For $y\in\Lambda_n$, write
\[
      \cL^y f:=c_y(\pi_yf-f),
      \qquad \textrm{so} \ \ \cL_n=\sum_{x\in \L_n}\cL^x.
\]
Since the state space is finite, it is straightforward to check that the killed process is generated on $L^2(\pi_n)$ by the self-adjoint operator
\begin{align}\label{eq:Hx}
             H^x:=\cL_n - \cL^x - c_x = \sum_{y\ne x}\cL^y-c_x\,.
\end{align}
In particular, $Q_t^x=e^{t H^x}$.



\begin{proposition}[Contraction of the killed-process]\label{prop:killed-contraction}
The lowest eigenvalue of $-H^x$ is uniformly positive. In particular, it is not smaller than $\kappa := \frac{1}{2}( p \wedge q) \l > 0$. Hence, for all $f \in L^2(\pi_n)$ and $t\ge 0$,
\begin{equation}\label{eq:coercivity}
       -\ip{f}{H^x f}_{\pi_n}\ge \kappa\|f\|_{2,\pi_n}^2,
       \qquad
       \|Q_t^x\|_{2\to2}\le e^{-\kappa t}.
\end{equation}
\end{proposition}

\begin{proof}

   By homogeneity, fix an eigenfunction $f$ of $-H^x$, such that $\|f\|_{2,\pi_n} = 1$.
   Applying Cauchy-Schwarz to $[Q_t^x f(\eta)]^2$ gives
   \begin{align*}
      \pi_n\left( [Q_t^x f]^2 \right) \leq \max_\eta \bbP_{\eta}(\tau_x\ge t)\leq \frac{1}{\min_{\eta}\pi_n(\eta)}\bbP_{\pi_n}(\tau_x\ge t)\,,
   \end{align*}
   where, in the first inequality, we used stationarity of $\pi_n$.
For $i\in\{0,1\}$, let $\tau_x^{i}$ be the hitting time of the set $\{\h\,:\, \h_x = i\}$,  and
observe that $\tau_x \leq \tau_x^0 \vee \tau_x^1$.
Thus, for any $t\ge 0$,
\begin{align*}
   \bbP_{\pi_n}(\tau_x\ge t) 
   &\le \bbP_{\pi_n}(\tau_x^0\ge t) + \bbP_{\pi_n}(\tau_x^1\ge t).
\end{align*}
By a standard result for reversible Markov chains, see for example \cite{Aldous-Fill}*{Lemma 3.35}, we have
\begin{align*}
   \bbP_{\pi_n}(\tau_x^1\ge t) \le e^{-\lambda^{(n)}\, p\, t}, \qquad \textrm{and} \qquad \bbP_{\pi_n}(\tau_x^0\ge t) \le e^{-\lambda^{(n)}\, q\, t}.
\end{align*}
Hence,
\begin{align*}
   \liminf_{t\to\infty} -\frac{1}{t}\log \pi_n\left( [Q_t^x f]^2 \right) \ge  ( p \wedge q) \l^{(n)} \ge ( p \wedge q) \l >0.
\end{align*}
Since $f$ is chosen to be an eigenfunction of $-H^x$, we have $Q_t^x f = e^{-t \mu} f$ for some eigenvalue $\mu$ of $-H^x$. 
Hence, the lowest eigenvalue of $-H^x$ is at least $\kappa = \frac{1}{2}( p \wedge q) \l >0$. 
\end{proof}

The legal rings at different sites are not independent under the stationary process.
However, since the killed process is a contraction in $L^2(\pi_n)$, it gives us control over the probability of observing no legal rings across an arbitrary collection of sites and time intervals.
This is made precise in the following lemma.

\begin{lemma}\label{lem:failure-string}
Fix $x_1,\ldots,x_r\in\Lambda_n$ (repeated sites are allowed) and a sequence of non-overlapping time intervals
\[
       0\le s_1<t_1\le s_2<t_2\le\cdots\le s_r<t_r\,,
\]
then
\begin{equation}\label{eq:failure-string}
       \bbP_{\pi_n}\left(\bigcap_{j=1}^r \{\text{no legal ring at $x_j$ during }[s_j,t_j)\}\right)
          \le \exp\left\{-\kappa\sum_{j=1}^r(t_j-s_j)\right\}.
\end{equation}
\end{lemma}

\begin{proof}
Let $P_t=e^{t \cL_n}$. The Markov property at the deterministic interval endpoints, and the definition of the killed semigroup give,
\begin{align*}
 \bbP_{\pi_n}\left(\bigcap_{j=1}^r\{\text{there is no legal ring at $x_j$ during }[s_j,t_j)\} \right) = \\
  =\ip{\mathbf{1}}{P_{s_1}Q_{t_1-s_1}^{x_1}P_{s_2-t_1}Q_{t_2-s_2}^{x_2}
       \cdots P_{s_r-t_{r-1}}Q_{t_r-s_r}^{x_r}\mathbf{1}}_{\pi_n}.
\end{align*}
Since $\|P_t\|_{2\to2}\le1$, and Proposition \ref{prop:killed-contraction} gives $\|Q_t^x\|_{2\to2}\le e^{-\kappa t}$, Cauchy--Schwarz and submultiplicativity yield \eqref{eq:failure-string}.
\end{proof}

\subsection{Monotone sequence of coupled sites}

The oriented nature of the update family allows us to define a monotone set of sites which have coupled with the stationary process, uniformly in the initial condition under the graphical construction. 

Let $\eta=\eta(0)$ be an arbitrary initial configuration and $(\eta(t))_{t\geq 0}$ the oriented KCM started from $\eta$. Recall the graphical construction of the process and the orientation $v$ in Definition \ref{def:oriented}. Notice that for each fixed site $x\in\Lambda_n$, the restriction $(\eta_{x+\bbH_v}(t))_{t\geq 0}$, of the process to the half-space $(x+\bbH_v)\cap\L_n$, is independent of the initial configuration and clock rings and coin tosses outside of $(x+\bbH_v)\cap\L_n$ (see \cite{HTbook}*{Proposition 7.3}). 
This is because the constraint at $x$ or `below' it (with respect to the direction $v$) depends only on the configuration `below' $x$.

Let $\sigma \sim \pi_n$, independently of the graphical construction, and $(\sigma(t))_{t\geq 0}$ be the associated stationary process. 
Fix a time step $s>0$ and let
\[
                  I_k:=[ks,(k+1)s),\qquad k=0,1,2,\ldots.
\]
For $x\in\Lambda_n$, denote the dependency neighbourhood of $x$ by
\[
             \cN_x:=(x+\cN)\cap\Lambda_n.
\]
All exterior sites in $\bbZ^d\setminus\Lambda_n$ are regarded as permanently coupled and facilitating. Starting from $G_0=\varnothing$, define
\begin{equation}\label{eq:Wrecursion}
 G_{k+1}:=G_k\cup
   \left\{x\in\Lambda_n:
       \cN_x\subseteq G_k\ \text{and $(\sigma(t))_{t\in I_k}$ has a legal ring at $x$ during $I_k$}
   \right\}.
\end{equation}
This defines an increasing sequence of sets of sites $G_k$ which are coupled with the stationary process after time $ks$. The following lemma shows that once a site is coupled, it remains coupled forever.

\begin{lemma}[Permanently coupled increasing sequence]\label{lem:permanent}
Suppose $\sigma \sim \pi_n$ independently of the graphical construction\footnote{To simplify notation, we include this initial independent configuration in $\bbP$.}. 
For each initial configuration $\eta \in \O_n$ let $(\eta(t))_{t\geq 0}$ and $(\sigma(t))_{t\geq 0}$ be driven by the same graphical construction. Then, for every $k\ge0$:
\begin{enumerate}
 \item $G_k\subseteq G_{k+1}$;
 \item if $x\in G_k\,$, then $\cN_x\subseteq G_k$;
 \item if $x\in G_k\,$, then $\eta_x(t)=\sigma_x(t)$ for every $t\ge ks$.
\end{enumerate}
\end{lemma}

\begin{proof}
$(1)$ is immediate by construction.
$(2)$ follows by induction; assume it holds for $G_k$, if $x \in G_{k+1}\setminus G_k$, then $\cN_x \subseteq G_k\subseteq G_{k+1}$.
 For $(3)$, fix $\eta \in \O_n$ an arbitrary initial configuration, suppose $(3)$ holds for $G_k$ and fix $x\in G_{k+1}\setminus G_k$. 
 By construction, $\cN_x\subseteq G_k$, so every site $y\in\cN_x$ satisfies $\eta_y(t)=\sigma_y(t)$ for every $t\ge ks$.
 Therefore, the constraints under $(\eta(t))_{t\geq 0}$ and $(\sigma(t))_{t\geq 0}$ at $x$ agree throughout $I_k$. 
 At the first legal ring of $(\sigma(t))_{t\geq 0}$ at $x$ in $I_k$, the same ring is legal for $(\eta(t))_{t\geq 0}$, and the two chains adopt the same spin at $x$. 
 All subsequent rings at $x$ are simultaneously legal or not in both chains, so equality at $x$ is conserved. 
 Sites already in $G_k$ remain coupled by induction.
\end{proof}

\subsection{Witnesses to failure}
\label{sec:witness}

Fix $x \in \L_n$, and suppose it has not coupled with the stationary process after $K>0$ time intervals. 
We define backwards witness paths which explain why $x$ has not coupled with the stationary process.
Each step backwards in the path either chooses a site in the dependency neighbourhood which had not yet coupled by the end of the previous time interval, or, if no such site exists, then stays at the same site.
If the path stays at the same site, then to avoid coupling, there must have been no legal ring in the stationary process at the site during the current time interval.
By the contraction of the killed process, the probability that there are many such ``stays'' is small.
Also, each move in the path must decrease the height by at least $\delta_v$.
Hence, the total number of moves is at most order $n\|v\|_1$, and the total number of stays is at least order $K-\textrm{const}\,n\|v\|_1 $.
Choosing $K$ a large multiple of $n$ completes the proof of the upper bound on the mixing time.

For $K\ge1$ and $x\in\Lambda_n$, let $\Gamma_K(x)$ be the set of deterministic sequences
\[
          \gamma=(z_0,\ldots,z_K),\qquad \textrm{with} \ \ z_K=x,
\]
such that for $0\le j<K$,
\[
      z_j=z_{j+1}
      \quad\text{or}\quad
      z_j=z_{j+1}+u\in\Lambda_n\text{ for some }u\in\cN.
\]
Let
\[
       S(\gamma):=\{j\in\{0,\ldots,K-1\}:z_j=z_{j+1}\}
\]
be the set of \emph{stays}  and call the steps $i \in \{0,\ldots,K-1\}\setminus S(\gamma)$ \emph{spatial moves}.  
Define
\begin{equation}\label{eq:Fgamma}
       F_\gamma:=\bigcap_{j\in S(\gamma)}
          \{\text{there is no legal ring at $z_{j+1}$ during $I_j$ under $(\sigma(t))_{t\geq 0}$}\}.
\end{equation}

Recall the orientation vector $v$ and $\delta_v = \min_{u\in \cN}(-\ip{u}{v})$ in Definition \ref{def:oriented}. 
For $z \in \L_n$, define the \emph{height} $h(z):=\ip{z}{v}$. 
Fix $x \in \L_n$ and $y \in \cN_x$, then $h(y) \le h(x) - \delta_v$. 
The height of the box is at most $n\|v\|_1$, hence every path in $\Gamma_K(x)$ has at most
\begin{equation}\label{eq:Rn}
       R_n:=\left\lceil\frac{n\|v\|_1}{\delta_v}\right\rceil \leq \rho_v n
\end{equation}
spatial moves, where $\rho_v = 1 + \|v\|_1/\delta_v$.

If $x\notin G_K$, then there exists a witness path $\gamma \in \Gamma_K(x)$ such that the corresponding failure event $F_\gamma$ must occur under the stationary process.
\begin{lemma}\label{lem:witness}
For every $K\ge1$ and $x\in\Lambda_n$,
\begin{equation}\label{eq:witness-inclusion}
             \{x\notin G_K\}\subseteq
                \bigcup_{\gamma\in\Gamma_K(x)}F_\gamma.
\end{equation}
Moreover every $\gamma\in\Gamma_K(x)$ has at most $R_n$ spatial moves and therefore at least $K-R_n$ stays.
\end{lemma}

\begin{proof}
Fix an outcome with $x\notin G_K$ and set $z_K=x$. 
We will construct a witness path $\gamma=(z_0,\ldots,z_K)\in\Gamma_K(x)$ backwards in time such that the corresponding failure event $F_\gamma$ occurs.
We proceed by induction.
Fix $j+1 \in \{1,\ldots,K\}$ and suppose $z_{j+1}\notin G_{j+1}$.
Since $G_j\subseteq G_{j+1}$, also $z_{j+1}\notin G_j$. Fix once and for all an ordering of $\cN$. 
If there is $u\in\cN$ such that $z_{j+1}+u\in\Lambda_n\setminus G_j$, choose the first such $u$ and set $z_j=z_{j+1}+u$. If no such move exists, then $\cN_{z_{j+1}}\subseteq G_j$. 
Since $z_{j+1}\notin G_{j+1}$, recursion \eqref{eq:Wrecursion} forces the absence of a legal ring at $z_{j+1}$ during $I_j$; set $z_j=z_{j+1}$. Thus each stay creates the corresponding failure event, proving \eqref{eq:witness-inclusion}. The move bound follows from \eqref{eq:Rn} and the preceding argument.
\end{proof}

\begin{proposition}\label{prop:incomplete}
For every $K\ge R_n$,
\begin{equation}\label{eq:incomplete}
       \bbP(G_K\ne\Lambda_n)
          \le n^d (1+N)^K\exp\{-\kappa s(K-R_n)\}.
\end{equation}
\end{proposition}

\begin{proof}
For a fixed $x\in \Lambda_n$, Lemma \ref{lem:witness} and a union bound give
\[
      \bbP(x\notin G_K)\le\sum_{\gamma\in\Gamma_K(x)}\bbP(F_\gamma).
\]
At each backward step there is one stay choice and at most $|\cN|$ move choices, so $|\Gamma_K(x)|\le (1+|\cN|)^K$. For a fixed candidate path, the sequence of sites and time intervals in \eqref{eq:Fgamma} are deterministic. By Lemma \ref{lem:failure-string},
\[
        \bbP(F_\gamma)\le e^{-\kappa s|S(\gamma)|}
             \le e^{-\kappa s(K-R_n)}.
\]
A union bound over the $n^d$ sites in $\Lambda_n$ proves \eqref{eq:incomplete}.
\end{proof}

\subsection{Proof of the upper bound}

If $G_K=\Lambda_n$, Lemma \ref{lem:permanent} implies that for each initial configuration $\eta$, we have $\eta(t)=\sigma(t)$ for every $t\ge Ks$. 
Since $\sigma_{Ks} \sim\pi_n$, the standard coupling inequality (see for example \cite{Levin2008}) gives, for $t \geq Ks$,
\[
   d_n(t)=\sup_\eta\|\bbP(\eta(t)\in \cdot)-\pi_n\|_{\textrm{TV}}
      \le \sup_\eta\bbP(\eta(t)\ne\sigma(t)) \leq \bbP(G_K\ne\Lambda_n).
\]
By Proposition \ref{prop:incomplete}, 
\begin{align*}
 \bbP(G_K\ne\Lambda_n)
  &\le \exp(d\log n-aK+\kappa sR_n) \leq\exp\left( (d+\kappa s  \rho_v)\, n-aK\right)\,,
\end{align*}
where $a := \kappa s - \log(1+|\cN|)$.
Choose $s>0$ so that $a > 0$ and
\[
K = \left\lceil \frac{d+\kappa s \rho_v + 2}{a}\right\rceil n\,.
\]
Then the right hand side is at most $\exp(-2n)$ for all $n$ sufficiently large, which completes the proof of the upper bound in \eqref{eq:main}.

\subsection{Proof of the lower bound}
This follows from a standard finite speed of propagation argument for the Poisson clock rings, see for example \cite{HTbook}*{Proposition 3.12} for a general statement and proof. The second part of \cite{HTbook}*{Proposition 3.12}, in the notation here, states that; there exists a constant $c=c(q,d,\cU)$ such that if $q > 0$, for all $\varepsilon \in (0,1)$, and all $n$ sufficiently large, 
\[
 t_{\mathrm{mix}}^{(n)}(\varepsilon) \ge c^{-1} n \,.
\]
The claimed bounds follow by taking worst case constants.

\subsection*{Acknowledgements}
I would like to thank Fabio Martinelli for kindly suggesting the current proof of Proposition \ref{prop:killed-contraction}, 
which is a significant improvement over an early draft using a proof generated by ChatGPT.
I would also like to thank Ivailo Hartarsky for helpful discussions.
After replacing the proof of Proposition \ref{prop:killed-contraction}, generative AI (ChatGPT, GPT-5.6) was used solely for linguistic proofreading, and any remaining errors are entirely my own.
For the purpose of open access, the author has applied a Creative Commons Attribution (CC BY) licence to any Author Accepted Manuscript version arising from this submission.

\bibliography{kcm}
\bibliographystyle{plain}
\end{document}